\documentclass[11pt,reqno]{amsart}
\usepackage{amssymb,latexsym}
\usepackage{graphicx}
\usepackage{color} 
\usepackage{url}		
\calclayout

\theoremstyle{plain}
\numberwithin{equation}{section}
\newtheorem{thm}{Theorem}[section]
\newtheorem{theorem}[thm]{Theorem}

\newtheorem{corollary}[thm]{Corollary}
\newtheorem{remark}[thm]{Remark}

\begin{document}
\setcounter{page}{1}

\title[New sums mixing harmonic numbers and central binomial coefficients]
{New sums mixing harmonic numbers and central binomial coefficients}
\author{Michel Bataille}
\address{Independent Researcher, 76520 Franqueville-Saint-Pierre, France}
\email{michelbataille@wanadoo.fr}
\medskip
\author{Robert Frontczak}
\address{Independent Researcher, 72764 Reutlingen, Germany}
\email{robert.frontczak@web.de}

\begin{abstract}
We study two new classes of sums with inverse binomial coefficients and harmonic numbers. In addition we establish recursive
solutions to the following power sums
\begin{equation*}
U_d(n) = \sum_{k=1}^n \frac{2^{2k}}{\binom{2k}{k}} \cdot k^d \quad \mbox{and}\quad  
V_d(n) = \sum_{k=1}^n \frac{2^{2k}}{\binom{2k}{k}}\cdot k^d\,H_k,
\end{equation*}
where $d$ is a positive integer.
\end{abstract}


\maketitle

\section{Introduction}

Combinatorial sums involving reciprocals of (central) binomial coefficients are prominent quantities.
They have been studied by many mathematicians in the past and even today there is no dwindling of interest for these sums 
\cite{Belbachir,Mansour,Sprugnoli,Sury,Witula}. A popular identity in this context is the following sum: 
\begin{equation}
\sum_{k=0}^n \frac{1}{\binom{n}{k}} = \frac{n+1}{2^{n+1}} \sum_{k=1}^{n+1} \frac{2^k}{k}.
\end{equation}
This identity was derived by Rockett in 1981 \cite{Rockett} but was known earlier to Comtet \cite[Page 294]{Comtet} 
and Gould \cite[Entry 2.25]{Gould}. For more information on this identity see the articles by Pla \cite{Pla}, Trif \cite{Trif}, 
Batir and his coauthors \cite{Batir1,Batir2,Batir3}, and Mansour who in \cite{Mansour} derived (among other things) an expression 
for the generalized sum of the form
\begin{equation*}
\sum_{k=0}^n \frac{a^k b^{n-k}}{\binom{n}{k}}, \quad a,b\in\mathbb{C}.
\end{equation*}
Another prominent combinatorial identity involving reciprocals of central binomial coefficients is (\cite{Adegoke1,Batir2,Sprugnoli,Witula})
\begin{equation}\label{central_bin0}
\sum_{k=0}^n  \frac{2^{2k}}{\binom{2k}{k}} = \frac{1}{3}\left ( (n+1)\frac{2^{2n+1}}{\binom{2n}{n}} + 1 \right ),
\end{equation}
which can be expressed equivalently in terms of Catalan numbers $C_n=\frac{1}{n+1}\binom{2n}{n}$ as
\begin{equation*}
\sum_{k=0}^n  \frac{2^{2k}}{\binom{2k}{k}} = \frac{1}{3}\left ( \frac{2^{2n+1}}{C_n} + 1 \right ).
\end{equation*}
Identity \eqref{central_bin0} was generalized recently by Bataille and Frontczak in \cite[Theorem 4.2]{Bataille}: 
For integers $m,n\ge 0$ we have 
\begin{equation}\label{central_bin}
\sum_{k=0}^n 2^{2k} \frac{\binom{m+k}{k}}{\binom{2k}{k}} 
= \frac{1}{2m+3}\left ( (m+n+1)\frac{\binom{m+n}{n}}{\binom{2n}{n}}2^{2n+1} + 1 \right ).
\end{equation}

Bataille and Frontczak proved \eqref{central_bin} by induction on $n$. A closer look shows, however, that this statement is actually true for all $m\in\mathbb C\setminus\mathbb Z^{-}$ and $m\neq -3/2$. This will be crucial in some of the proofs below.\\

In this paper, we derive identities for two (probably) new classes of combinatorial sums involving harmonic numbers 
that are associated with \eqref{central_bin}. For instance, we will show that the harmonic sum counterparts of \eqref{central_bin0}
are given by 
\begin{equation}\label{central_bin0_har}
\sum_{k=0}^n \frac{2^{2k}}{\binom{2k}{k}} H_k = - \frac{2}{9} + \frac{2^{2n+1}}{3\binom{2n}{n}}\left ( (n+1)H_n - \frac{2n-1}{3} \right ) 
\end{equation}
and
\begin{equation}\label{central_bin_evenhar}
\sum_{k=0}^n \frac{2^{2k}}{\binom{2k}{k}} H_{2k} = \frac{1}{9} + \frac{2^{2n}}{3 \binom{2n}{n}}\left (2(n+1)H_{2n} - \frac{4n+1}{3}\right ),
\end{equation}
both seeming to be new results. Moreover, we establish recursion formulas for the sums
\begin{equation*}
U_d(n) = \sum_{k=1}^n \frac{2^{2k}}{\binom{2k}{k}} \cdot k^d \quad \mbox{and}\quad  
V_d(n) = \sum_{k=1}^n \frac{2^{2k}}{\binom{2k}{k}}\cdot k^d\,H_k,
\end{equation*}
where $d$ is a positive integer. \\

We conclude the introductory section with some definitions and preliminary observations. \\
For complex numbers $r$ and $s$, the generalized binomial coefficients are defined by
\begin{equation*}
\binom {r}{s} = \frac{\Gamma (r+1)}{\Gamma (s+1) \Gamma (r-s+1)},
\end{equation*}
where the Gamma function, $\Gamma(z)$, is defined for $\Re(z)>0$ by the integral \cite{Srivastava}
\begin{equation*}
\Gamma (z) = \int_0^\infty e^{- t} t^{z - 1}\,dt.
\end{equation*}
The function $\Gamma(z)$ can be extended to the whole complex plane by analytic continuation. 
It has a simple pole at each of the points $z=\cdots,-3,-2,-1,0$. 
The Gamma function extends the classical factorial function to the complex plane by $(z-1)!=\Gamma(z)$,
which shows that, if $r$ and $s$ are non-negative integers then we get the usual binomial coefficients
\begin{equation*}
\binom {r}{s} =
\begin{cases}
\dfrac{r!}{s!(r - s)!}, & \text{$r\geq s$};\\
0, & \text{$r<s$}.
\end{cases}
\end{equation*} 

Harmonic numbers $H_z$ and odd harmonic numbers $O_z$ are defined for $0\ne z\in\mathbb C\setminus\mathbb Z^{-}$ by the recurrence relations
\begin{equation*}
H_z = H_{z - 1} + \frac{1}{z} \qquad \text{and} \qquad O_z = O_{z - 1} + \frac{1}{2z - 1},
\end{equation*}
with $H_0=0$ and $O_0=0$. Harmonic numbers are connected to the digamma function through the fundamental relation \cite{Srivastava}
\begin{equation}\label{Har_psi}
H_z = \psi(z + 1) + \gamma,
\end{equation}
where $\gamma$ is the Euler-Mascheroni constant and $\psi(z)=\Gamma'(z)/\Gamma(z)$ is the digamma (or psi) function. 
The fundamental relation \eqref{Har_psi} shows that harmonic numbers must be interpreted in general as digamma functions.
In particular, at rational arguments $z=p/q$ harmonic numbers may be evaluated using the corresponding digamma (or polygamma) values.
When each $z$ is a positive integer, say $n$, we have the sequences of harmonic and odd harmonic numbers $H_n$ and $O_n$, respectively,
and the recurrence relations give
\begin{equation*}
H_n =\sum_{k=1}^n \frac{1}{k} \qquad \text{and} \qquad O_n =\sum_{k=1}^n \frac{1}{2k-1}.
\end{equation*}
Obvious relations between harmonic numbers $H_n$ and odd harmonic numbers $O_n$ are given by
\begin{equation}\label{eq.har_oddhar}
H_{2n} = \frac{1}{2} H_n + O_n \qquad \text{and} \qquad H_{2n - 1} = \frac{1}{2}H_{n - 1} + O_n.
\end{equation}
More such relations exist and will be used later.  

\section{A first result}

In this section, we present the first main finding of the paper.

\begin{theorem}\label{main_thm1}
For all $n\geq 0$ and all $m\in\mathbb C\setminus\mathbb Z^{-}$ and $m\neq -3/2,$ we have
\begin{align}
\sum_{k=0}^n 2^{2k} \frac{\binom{m+k}{k}}{\binom{2k}{k}} H_{k+m} &= \frac{1}{2m+3}\left ( 2^{2n+1} (m+n+1)\frac{\binom{m+n}{n}}{\binom{2n}{n}} H_{n+m} + H_m \right ) \nonumber \\
&\qquad - \frac{2}{(2m+3)^2} \left (2^{2n} (2n-1)\frac{\binom{m+n}{n}}{\binom{2n}{n}} + 1\right ).
\end{align}
\end{theorem}
\begin{proof}
Differentiate \eqref{central_bin} w.r.t. $m$ and use
$$\frac{d}{dm} \binom{m+k}{k} = \binom{m+k}{k} (H_{k+m}-H_m).$$
When simplifying use \eqref{central_bin} again.
\end{proof}

Sums "inverse in structure" and alternating in signs can be found in the articles by Chu \cite{Chu} and Jin and Du \cite{Jin}.
For instance, we have the evaluation
\begin{equation*}
\sum_{k=0}^n (-1)^k \frac{\binom{n}{k}}{\binom{m+k}{k}} H_k = \frac{m}{n+m} \left ( H_{m-1} - H_{n+m-1} \right ).
\end{equation*}

\begin{corollary}
For all $n\geq 0$ we have
\begin{equation}\label{central_bin_har1}
\sum_{k=0}^n \frac{2^{2k}}{\binom{2k}{k}} H_k = \frac{2^{2n+1}}{3\binom{2n}{n}}\left ( (n+1)H_n - \frac{2n-1}{3} \right ) - \frac{2}{9},
\end{equation}
and
\begin{equation}\label{central_bin_har2}
\sum_{k=0}^n \frac{2^{2k}}{\binom{2k}{k}} \binom{n+k}{k} H_{n+k} = \frac{2^{2n+1}}{2n+3}\left ( (2n+1)H_{2n} - \frac{2n-1}{2n+3} \right )
+ \frac{1}{2n+3}\left ( H_{n} - \frac{2}{2n+3} \right ).
\end{equation}
\end{corollary}
\begin{proof}
Set $m=0$ and $m=n$, in turn, in Theorem \ref{main_thm1} and simplify.
\end{proof}

\begin{remark}
We note that the odd harmonic number counterpart of identity \eqref{central_bin_har1} is known. It was discovered recently independently by 
Campbell, and Batir and Sofo:
\begin{equation}\label{central_bin_ohar}
\sum_{k=0}^n \frac{2^{2k}}{\binom{2k}{k}} O_k = \frac{2}{9} + \frac{2}{9}\frac{2^{2n}}{\binom{2n}{n}}(n+1)(3O_n - 1).
\end{equation}
See Example 11 and Remark 3 in \cite{Batir2}. This can be used to establish a closed form for the sum containing even harmonic numbers.
\end{remark}

\begin{corollary}
For all $n\geq 0$ we have
\begin{equation}\label{central_bin_evenhar}
\sum_{k=0}^n \frac{2^{2k}}{\binom{2k}{k}} H_{2k} = \frac{1}{9} + \frac{2^{2n}}{3 \binom{2n}{n}}\left (2(n+1)H_{2n} - \frac{4n+1}{3}\right ).
\end{equation}
\end{corollary}
\begin{proof}
Combine \eqref{central_bin_har1} with \eqref{central_bin_ohar} via the first part of \eqref{eq.har_oddhar}.
\end{proof}

Two other special cases of Theorem \ref{main_thm1} allow us to rediscover two summation identities involving odd harmonic numbers.

\begin{corollary}
For all $n\geq 0$ we have
\begin{equation}\label{central_bin_ohar1}
\sum_{k=0}^n (2k+1) O_{k+1} = \frac{1}{4}\left ( (2n+1)(2n+3)O_{n+1} - (n-1)(n+1) \right )
\end{equation}
and
\begin{equation}\label{central_bin_ohar2}
\sum_{k=0}^n O_{k} = \left (n+\frac{1}{2}\right ) O_{n} - \frac{n}{2}.
\end{equation}
\end{corollary}
\begin{proof}
To obtain the first sum set $m=1/2$ in Theorem \ref{main_thm1}. The binomial coefficient relation
$$\binom{u+1/2}{v} = 2^{-2v} \binom{2u+1}{2v} \frac{\binom{2v}{v}}{\binom{u}{v}}$$ 
yields
$$\frac{\binom{k+1/2}{k}}{\binom{2k}{k}} = 2^{-2k} (2k+1)$$
and
$$\sum_{k=0}^n (2k+1) H_{k+1/2} = \frac{1}{4}\left ( (2n+1)(2n+3)H_{n+1/2} + H_{1/2} \right ) -\frac{1}{8} \left ((2n-1)(2n+1)+1 \right ). 
$$
Next, use (see \cite{Adegoke0} for a proof)
$$H_{n+1/2} = 2O_{n+1} - 2 \ln(2),$$
and the result follows by comparing the rational and irrational parts. Similarly, with $m=-1/2$, the second sum is obtained using
$$\binom{u-1/2}{v} = 2^{-2v} \binom{u}{v} \frac{\binom{2u}{u}}{\binom{2(u-v)}{u-v}}$$
and again comparing the rational and irrational parts.
\end{proof}

\section{Two recursive results}

In this section, we study other generalizations of formulas \eqref{central_bin0} and \eqref{central_bin0_har} by establishing recursion formulas for the sums
\begin{equation*}
U_d(n) = \sum_{k=1}^n \frac{2^{2k}}{\binom{2k}{k}} \cdot k^d \quad \mbox{and}\quad  
V_d(n) = \sum_{k=1}^n \frac{2^{2k}}{\binom{2k}{k}}\cdot k^d\,H_k,
\end{equation*}
where $d$ is a positive integer. We first prove the following theorem.

\begin{theorem}\label{main_thm_rec1} 
Let $d,n$ be positive integers. Then, 
\begin{equation}
(2d+3)U_d(n) = \frac{(n+1)2^{2n+1}n^d}{\binom{2n}{n}} - \sum_{j=1}^d (-1)^j c_{d,j} U_{d-j}(n)
\end{equation}
with
$$U_0(n) = \frac{1}{3}\left( \frac{(n+1)2^{2n+1}}{\binom{2n}{n}}-2\right),$$
and
\begin{equation}\label{c_term}
c_{d,j} = \binom{d+1}{j+1} + \binom{d}{j+1}.
\end{equation}
\end{theorem}
\begin{proof}
We adopt the following notations
$$u_k = \frac{2^{2k}}{\binom{2k}{k}},\qquad  a_k = (2k-1)(k-1)^d.$$
Since $u_0=1$, the above value of $U_0(n)$ readily follows from \eqref{central_bin0}. Next, using 
$$\binom{2k+2}{k+1} = \frac{2(2k+1)}{k+1}\binom{2k}{k}$$
a short calculation gives
$$(2k+1)(u_{k+1}-u_k) = u_k,$$
from which we deduce
$$U_d(n) = \sum_{k=0}^n k^d u_k = \sum_{k=0}^n k^d (2k+1)(u_{k+1} - u_k) = \sum_{k=0}^n a_{k+1}(u_{k+1}-u_k).$$
With the help of Abel's transformation (i.e. summation by parts), we obtain
$$U_d(n) = a_{n+1}u_{n+1} - \sum_{k=1}^n (a_{k+1}-a_k) u_k.$$
We have
\begin{align*} 
a_{k+1}-a_k&=(2k+1)k^d-(2k-1)\sum_{j=0}^d\binom{d}{j}(-1)^j k^{d-j}\\
&=(2k+1)k^d-(2k-1)k^d-(2k-1)\sum_{j=1}^d\binom{d}{j}(-1)^j k^{d-j}\\
&=(2d+2)k^d-2\sum_{j=2}^d \binom{d}{j}(-1)^j k^{d-j+1}+\sum_{j=1}^d\binom{d}{j}(-1)^j k^{d-j}\\
&=(2d+2)k^d+\sum_{j=1}^d\left(2\binom{d}{j+1}+\binom{d}{j}\right)(-1)^jk^{d-j}\\
&=(2d+2)k^d+\sum_{j=1}^d c_{d,j}(-1)^j k^{d-j}.
\end{align*}
Therefore
$$U_d(n) = a_{n+1}u_{n+1} - (2d+2) U_d(n) - \sum_{j=1}^d c_{d,j}(-1)^j U_{d-j}$$
and the theorem immediately follows.
\end{proof}

A straightforward calculation leads to the explicit values of $U_1(n), U_2(n)$ and $U_3(n)$:

\begin{corollary}
For any positive integer $n$, we have
\begin{equation}
\sum_{k=1}^n \frac{2^{2k}}{\binom{2k}{k}}\cdot k = \frac{1}{15} \left(\frac{(3n+1)(n+1)2^{2n+1}}{\binom{2n}{n}}-2\right),
\end{equation}
\begin{equation}
\sum_{k=1}^n \frac{2^{2k}}{\binom{2k}{k}}\cdot k^2 = \frac{1}{105}\left(\frac{(15n^2+12n-1)(n+1)2^{2n+1}}{\binom{2n}{n}}+2\right),
\end{equation}
and
\begin{equation}
\sum_{k=1}^n \frac{2^{2k}}{\binom{2k}{k}}\cdot k^3 = \frac{1}{945}\left(\frac{(105n^3+135n^2+3n-9)(n+1)2^{2n+1}}{\binom{2n}{n}}+18\right).
\end{equation}
\end{corollary}

Our next theorem shows how to calculate $V_d(n)$ \textit{via} a recursion formula and the value of $U_d(n)$.

\begin{theorem}\label{main_thm_rec2} 
Let $d,n$ be positive integers. Then, 
\begin{equation}
(2d+3)V_d(n) = \frac{(n+1)2^{2n+1} n^d H_{n+1}}{\binom{2n}{n}} - \sum_{j=1}^d (-1)^j c_{d,j}V_{d-j}(n) - 2U_d(n),
\end{equation}
where the coefficients $c_{d,j}$ are given in \eqref{c_term} and where the value of $V_0(n)$ is provided by \eqref{central_bin0_har}.
\end{theorem} 
\begin{proof}
Again, let $a_k=(2k-1)(k-1)^d$ and $c_{d,j}$ be defined as in Theorem \ref{main_thm_rec1}. Setting $v_k = \frac{2^{2k}}{\binom{2k}{k}}H_k$, 
we easily obtain 
$$(2k+1)(v_{k+1}-v_k) = v_k + 2u_k$$ 
and deduce that
$$V_d(n) = \sum_{k=0}^n k^d v_k = \sum_{k=0}^n a_{k+1}(v_{k+1}-v_k) - 2\sum_{k=1}^n u_k k^d = a_{n+1}v_{n+1} 
- \sum_{k=1}^n(a_{k+1}-a_k) v_k - 2U_d(n).$$
The theorem now follows from
$$V_d(n) = a_{n+1}v_{n+1} - (2d+2)V_d(n) - \sum_{j=1}^d (-1)^j c_{d,j}V_{d-j}(n) - 2U_d(n).$$
\end{proof}

As a corollary, we mention the values of $V_1(n)$ and $V_2(n)$:
\begin{corollary}
We have
\begin{equation}
V_1(n) = \frac{(3n+1)(n+1)2^{2n+1}H_n}{15\binom{2n}{n}}-\frac{(18n^2-11n+1)2^{2n+1}}{225\binom{2n}{n}}+\frac{2}{225}
\end{equation}
and
\begin{equation}
V_2(n) = \frac{(n+1)(15n^2+12n-1)2^{2n+1}H_n}{105\binom{2n}{n}}-\frac{(450n^3-261n^2-328n+173)2^{2n+1}}{11025\binom{2n}{n}}+\frac{346}{11025}.
\end{equation}
\end{corollary}

\begin{remark}\label{PQ-Remark}
From \eqref{central_bin0_har} and noting that
\begin{equation*}
V_1(n) = \frac{2^{2n+1}}{15 \binom{2n}{n}}\left ((n+1)(3n+1)H_n -\frac{(2n-1)(9n-1)}{15}\right ) + \frac{2}{225}
\end{equation*}
and
\begin{equation*}
V_2(n) = \frac{2^{2n+1}}{105 \binom{2n}{n}}\left ((n+1)(15n^2+12n+1)H_n -\frac{(2n-1)(225n^2-18n-173)}{105}\right ) + \frac{346}{11025}
\end{equation*}
we recognize that $V_d(n)$ can be written as
\begin{equation}
\sum_{k=1}^n \frac{2^{2k}}{\binom{2k}{k}} \cdot k^d \,H_k 
= \frac{2^{2n+1}}{N(d) \binom{2n}{n}} \left ( (n+1)P_d(n)H_n -\frac{(2n-1)Q_d(n)}{N(d)}\right ) + \frac{C(d)}{N^2(d)},
\end{equation}
with polynomials $P_d(n)$ and $Q_d(n)$ both of degree $d$, $C(d)$ a constant term and the normalizing factor $N(d)$ given by
$$N(d) = \prod_{j=0}^{d+1} (2j+1).$$
An analogous observation can be made for $U_d(n)$.
\end{remark}

\section{A second family of harmonic sums}

\begin{theorem}\label{main_thm2}
For all $n\geq 1$ and all $m\in\mathbb C\setminus\mathbb Z^{-}$ and $m\neq -3/2$ we have
\begin{align}
\sum_{k=1}^n \frac{2^{-2k}}{2k-1} \frac{\binom{2k}{k}}{\binom{m+k+1}{k}} H_{m+k+1} &= \frac{4m+5}{(m+1)(2m+3)^2} + \frac{H_m}{2m+3} \nonumber \\
&\qquad - \frac{2^{-2n}}{2m+3} \frac{\binom{2n}{n}}{\binom{m+n+1}{n}}\left (H_{m+n+1} + \frac{2}{2m+3} \right ).
\end{align}
\end{theorem}
\begin{proof}
Bataille and Frontczak have shown in \cite[Theorem 4.2]{Bataille} the combinatorial identity
\begin{equation}\label{central_bin2}
\sum_{k=0}^{n-1} \frac{2^{2k}}{2(n-k)-1} \frac{\binom{m+n+1}{k} \binom{2(n-k)}{n-k}}{\binom{n}{k}} 
= 2^{2n} \frac{m+n+1}{(m+1)(2m+3)}\binom{m+n}{n} - \frac{1}{2m+3}\binom{2n}{n}.
\end{equation}
Differentiate w.r.t. $m$ and use
$$\frac{d}{dm} \binom{m+n+1}{k} = \binom{m+n+1}{k} (H_{m+n+1}-H_{m+n-k+1}).$$
When simplifying use \eqref{central_bin2} again. After changing the summation index via
$$\sum_{k=0}^{n-1} a_k = \sum_{k=1}^{n} a_{n-k},$$
this will result in
\begin{align}\label{interim}
&\sum_{k=1}^n \frac{2^{-2k}}{2k-1} \frac{\binom{m+n+1}{m+k+1} \binom{2k}{k}}{\binom{n}{k}} H_{m+k+1} = 
\frac{(m+n+1)}{(m+1)(2m+3)} \binom{m+n}{n} \left ( H_{n+m+1} - H_{n+m} + H_m \right ) \nonumber \\
&\qquad + \frac{n(4m+5)+2(m+1)^2}{(m+1)^2(2m+3)^2} \binom{m+n}{n} - \frac{2^{-2n}}{2m+3} \binom{2n}{n}\left (H_{n+m+1} + \frac{2}{2m+3} \right ).
\end{align}
The final step in to make use of
$$\binom{u+1}{v+1} = \frac{u+1}{v+1} \binom{u}{v},$$
write the left-hand side of \eqref{interim} as
$$\binom{m+n+1}{n} \sum_{k=1}^n \frac{2^{-2k}}{2k-1} \frac{\binom{2k}{k}}{\binom{m+k+1}{k}} H_{m+k+1}$$
simplify further keeping in mind that
$$\frac{\binom{m+n}{n}}{\binom{m+n+1}{n}} = \frac{m+1}{m+n+1}.$$
\end{proof}

\begin{corollary}
We have
\begin{equation}
\sum_{k=1}^n \frac{2^{-2k}}{2k-1} \frac{\binom{2k}{k}}{k+1} H_{k+1} 
= \frac{5}{9} - \frac{2^{-2n}}{3(n+1)} \binom{2n}{n} \left ( H_{n+1} + \frac{2}{3}\right )
\end{equation}
and
\begin{equation}
\sum_{k=1}^n \frac{2^{-2k}}{2k-1} \frac{\binom{2k}{k}}{(k+1)(k+2)} H_{k+2} 
= \frac{19}{100} - \frac{2^{-2n}}{5(n+1)(n+2)} \binom{2n}{n} \left ( H_{n+2} + \frac{2}{5}\right ).
\end{equation}
\end{corollary}
\begin{proof}
Set $m=0$ and $m=1$ in turn in Theorem \ref{main_thm2}.
\end{proof}

\begin{corollary}
We have
\begin{equation}
\sum_{k=1}^n \frac{2^{-2k}}{2k-1} \frac{\binom{2k}{k}}{\binom{n+k}{k}} H_{n+k} 
= \frac{1}{2n+1}\left (H_n + \frac{2}{2n+1}\right ) - \frac{2^{-2n}}{2n+1} \left (H_{2n} + \frac{2}{2n+1}\right ).
\end{equation}
\end{corollary}
\begin{proof}
Set $m=n-1$ in Theorem \ref{main_thm2}.
\end{proof}

The last identity reminds us in a way of
\begin{equation}\label{Riordan_id}
\sum_{k=0}^n \frac{2^{-2k}}{2k-1} \binom{2k}{k} = - 2^{-2n} \binom{2n}{n},
\end{equation}
which can be found in Riordan's book \cite[p.130]{Riordan}. More identities of this nature can be found in \cite{Adegoke0}, for instance.

\section{A final identity}

Our final identity involves harmonic numbers of order 2, that is, $H_n^{(2)}=\sum_{k=1}^n 1/k^2$. 

\begin{theorem}
For all $n\geq 1$ and all $m\in\mathbb C\setminus\mathbb Z^{-}$ and $m\neq -3/2$ we have
\begin{equation}\label{eq_har_final}
\sum_{k=0}^n 2^{2k} \frac{\binom{m+k}{k}}{\binom{2k}{k}}\left ( H_{k+m}^2 - H_{k+m}^{(2)}\right ) 
= \frac{A(m,n)}{2m+3} - \frac{B(m,n)}{(2m+3)^2} + 8 \frac{C(m,n)}{(2m+3)^3},
\end{equation}
with
\begin{equation}
A(m,n) = H_m^2-H_m^{(2)} + 2^{2n+1} \frac{\binom{m+n}{n}}{\binom{2n}{n}} \left ( H_{m+n}+(m+n+1)\left (H_{m+n}^2-H_{m+n}^{(2)}\right )\right ),
\end{equation}
\begin{equation}
B(m,n) = 2^{2n+2} (m+n+1) \frac{\binom{m+n}{n}}{\binom{2n}{n}}H_{m+n} + 4H_m + 2^{2n+1}(2n-1) \frac{\binom{m+n}{n}}{\binom{2n}{n}} H_{m+n},
\end{equation}
and
\begin{equation}
C(m,n) = 2^{2n} (2n-1) \frac{\binom{m+n}{n}}{\binom{2n}{n}} + 1.
\end{equation}
In particular, we have
\begin{align}\label{eq_har_final}
\sum_{k=0}^n \frac{2^{2k}}{\binom{2k}{k}}\left ( H_{k}^2 - H_{k}^{(2)}\right ) 
&= \frac{2^{2n+1}}{3 \binom{2n}{n}}\left ((n+1)\left( H_{n}^2 - H_{n}^{(2)}\right ) -  \frac{2}{3} (2n-1) H_n \right ) \nonumber \\
&\qquad + \frac{8}{27}\left ((2n-1) \frac{2^{2n}}{\binom{2n}{n}} + 1 \right).
\end{align}
\end{theorem}
\begin{proof}
From \eqref{Har_psi} we get
$$\frac{d}{dz} H_z = \frac{d}{dz} \psi(z+1) = \sum_{j=0}^\infty \frac{1}{(j+1+z)^2} = \zeta(2) - H_z^{(2)},$$
where $H_z^{(2)}$ is defined by $H_0^{(2)}=0$ and $H_z^{(2)}=H_{z-1}^{(2)}+\frac{1}{z^2}$.\\
We differentiate the identity in Theorem \ref{main_thm1} w.r.t. the parameter $m$. As
$$\frac{d}{dm} \binom{m+k}{k}H_{k+m} = \binom{m+k}{k}\left (H_{m+k}(H_{m+k}-H_m) - H_{m+k}^{(2)} + \frac{\pi^2}{6}\right ),$$
the left hand side is seen to be
$$\sum_{k=0}^n 2^{2k} \frac{\binom{m+k}{k}}{\binom{2k}{k}}\left ( H_{k+m}^2 - H_{k+m}^{(2)}\right ) - H_m \sum_{k=0}^n 2^{2k} \frac{\binom{m+k}{k}}{\binom{2k}{k}}H_{k+m} + \frac{\pi^2}{6} \sum_{k=0}^n 2^{2k} \frac{\binom{m+k}{k}}{\binom{2k}{k}}.$$ 
The second sum is evaluated using Theorem \ref{main_thm1} a second time while the third sum follows from \eqref{central_bin}. Differentiating the right hand side of the identity in Theorem \ref{main_thm1} is straightforward but lengthy. We have
\begin{align*}
&\frac{d}{dm} \left (\frac{1}{2m+3}\left ( 2^{2n+1} (m+n+1)\frac{\binom{m+n}{n}}{\binom{2n}{n}} H_{n+m} + H_m \right )\right ) \\ 
&= - \frac{2}{(2m+3)^2} \left( 2^{2n+1} (m+n+1)\frac{\binom{m+n}{n}}{\binom{2n}{n}} H_{n+m} + H_m \right ) 
+ \frac{1}{2m+3}\left (\frac{\pi^2}{6} - H_m^{(2)} \right ) \\
& + \frac{2^{2n+1}}{(2m+3)\binom{2n}{n}}\left (\binom{m+n}{n} H_{m+n} + (m+n+1)\binom{m+n}{n}\left (H_{m+n}(H_{m+n}-H_m) 
+ \frac{\pi^2}{6} - H_{m+n}^{(2)} \right )\right )
\end{align*}
and
\begin{align*}
\frac{d}{dm} \left (\frac{2}{(2m+3)^2} \left (2^{2n} (2n-1)\frac{\binom{m+n}{n}}{\binom{2n}{n}} + 1\right )\right )
&= \frac{1}{(2m+3)^2} 2^{2n+1} (2n-1) \frac{\binom{m+n}{n}}{\binom{2n}{n}}( H_{n+m} - H_m ) \\
&\qquad - \frac{8}{(2m+3)^3} \left (2^{2n} (2n-1)\frac{\binom{m+n}{n}}{\binom{2n}{n}} + 1\right ).
\end{align*}
We see that the parts containing $\pi^2/6$ cancel out and the final result follows after some steps of simplifications.
\end{proof}

\begin{corollary}
For all $n\geq 1$ we have
\begin{equation}\label{H_squared}
\sum_{k=1}^n \frac{2^{2k}}{\binom{2k}{k}} H_{k}^2 = \frac{44}{27} + \frac{2^{2n+1}}{3 \binom{2n}{n}}\left ((n+1)H_{n}^2 
- \frac{2(2n-1)}{3}H_n + \frac{8n-22}{9} \right ) + \frac{1}{3} \sum_{k=1}^n \frac{2^{2k}}{k^2\,\binom{2k}{k}}
\end{equation}
and also
\begin{equation}\label{H_secorder}
\sum_{k=1}^n \frac{2^{2k}}{\binom{2k}{k}} H_{k}^{(2)} = \frac{4}{3} + \frac{2^{2n+1}}{3 \binom{2n}{n}}\left ((n+1)H_{n}^{(2)} - 2 \right ) 
+ \frac{1}{3} \sum_{k=1}^n \frac{2^{2k}}{k^2\,\binom{2k}{k}}.
\end{equation}
\end{corollary}
\begin{proof}
We calculate as follows
\begin{align*}
\sum_{k=1}^n \frac{2^{2k}}{\binom{2k}{k}} H_{k}^{(2)} &= \sum_{k=1}^n \frac{2^{2k}}{\binom{2k}{k}}\sum_{j=1}^k \frac{1}{j^2} 
= \sum_{j=1}^n \frac{1}{j^2} \sum_{k=j}^n \frac{2^{2k}}{\binom{2k}{k}} \\
&= \sum_{j=1}^n \frac{1}{j^2} \left (\sum_{k=1}^n \frac{2^{2k}}{\binom{2k}{k}} - \sum_{k=1}^{j-1} \frac{2^{2k}}{\binom{2k}{k}} \right ) \\
&= \frac{1}{3}\frac{2^{2n+1}}{C_n} H_n^{(2)} - \frac{1}{3}\sum_{j=1}^n \frac{1}{j^2} \frac{2^{2j-1}}{C_{j-1}} \\
&= \frac{1}{3}\frac{2^{2n+1}}{C_n} H_n^{(2)} - \frac{2}{3}\sum_{j=1}^n \frac{2^{2j}}{j \binom{2j}{j}} 
+ \frac{1}{3}\sum_{j=1}^n \frac{2^{2j}}{j^2 \binom{2j}{j}}, \\
\end{align*}
where in the last step we have used the recursion
$$C_{n-1}= \frac{(n+1)C_n}{2(2n-1)}.$$
The expression for the first sum
$$\sum_{j=1}^n \frac{2^{2j}}{j \binom{2j}{j}} = 2 \left (\frac{2^{2n}}{\binom{2n}{n}} - 1 \right )$$
is known as Parker's formula (see \cite{Witula} or Theorem 4.1 in \cite{Sprugnoli}) and we get \eqref{H_secorder}. 
Inserting \eqref{H_secorder} into \eqref{eq_har_final} gives \eqref{H_squared} and completes the proof.
\end{proof}

\begin{remark}
The harmlessly looking sum 
$$\sum_{j=1}^n \frac{2^{2j}}{j^2 \binom{2j}{j}}$$
is hard to evaluate. We could not find it in the literature and it escaped all of our efforts.
\end{remark}

\section{Conclusion}

This paper is an addendum to \cite{Bataille}. We have applied the well-known derivative operator approach to derive expressions 
for seemingly new types of sums involving harmonic numbers and (inverse) central binomial coefficients. 
The identities \eqref{central_bin_har1}, \eqref{central_bin_evenhar} and \eqref{eq_har_final} which are special cases of our more general findings highlight our study. Moreover, we have established recursive solutions for the power sums
\begin{equation*}
U_d(n) = \sum_{k=1}^n \frac{2^{2k}}{\binom{2k}{k}} \cdot k^d \quad \mbox{and}\quad  
V_d(n) = \sum_{k=1}^n \frac{2^{2k}}{\binom{2k}{k}}\cdot k^d\,H_k,
\end{equation*}
where $d$ is a positive integer. It could be interesting to study the polynomials $P_d(n)$ and $Q_d(n)$ appearing in 
Remark \ref{PQ-Remark} further. As another line of research it seems worthwhile to move into the same direction as 
Mansour in \cite{Mansour} and Batir and Sofo in \cite{Batir2} who derived expressions for the polynomials
\begin{equation*}
\sum_{k=0}^n \frac{x^k}{\binom{n}{k}} \qquad \mbox{and}\qquad  \sum_{k=0}^n \frac{2^{2k}}{\binom{2k}{k}}\,x^k.
\end{equation*}
Doing so, a better understanding of the polynomials
\begin{equation*}
\sum_{k=0}^n \frac{2^{2k}}{\binom{2k}{k}}\,H_k\,x^k,  \qquad \sum_{k=0}^n \frac{2^{2k}}{\binom{2k}{k}}\,O_k\,x^k, \qquad\mbox{and}\qquad  
\sum_{k=0}^n \frac{2^{2k}}{\binom{2k}{k}}\,H_{2k}\,x^k
\end{equation*}
is desirable.


\end{document}